\documentclass[12pt,twoside]{amsart}
\usepackage{geometry}
\usepackage{amssymb,amsmath,amsthm, amscd, enumerate, mathrsfs}
\usepackage{graphicx, hhline}
\usepackage[all]{xy}
\usepackage{color}

\usepackage{iftex}

\ifPDFTeX
  \usepackage{hyperref}
\else
  \usepackage[dvipdfmx]{hyperref}
\fi

\title{Rational chain connectedness in the complex analytic setting}
\author{Osamu Fujino}
\date{2026/10/17, version 0.17}
\address{Department of 
Mathematics, Graduate School of Science, 
Kyoto University, Kyoto 606-8502, Japan}
\email{fujino@math.kyoto-u.ac.jp}
\DeclareMathOperator{\Nklt}{Nklt}
\DeclareMathOperator{\Supp}{Supp}
\DeclareMathOperator{\Exc}{Exc}

\newtheorem{thm}{Theorem}[section]
\newtheorem{lem}[thm]{Lemma}
\newtheorem{cor}[thm]{Corollary}
\newtheorem{prop}[thm]{Proposition}

\theoremstyle{definition}
\newtheorem{defn}[thm]{Definition}
\newtheorem{rem}[thm]{Remark}
\newtheorem{ex}[thm]{Example}
\newtheorem*{ack}{Acknowledgments}  
\newtheorem{step}{Step}

\makeatletter
  
  \@addtoreset{equation}{section}
\makeatother
\begin{document}

\begin{abstract}
We extend Hacon--M\textsuperscript{c}Kernan's 
rational chain connectedness theorem to the 
complex analytic setting. 
As a consequence, we prove that the fibers of any resolution of singularities 
of complex analytic kawamata log terminal singularities 
are rationally chain connected. 
In contrast to the original approach, we avoid the use of extension theorems 
and instead rely on the minimal model program.
\end{abstract}

\maketitle 

\tableofcontents 

\section{Introduction}\label{p-sec1}

In this paper, we generalize the main results of \cite{hacon-mckernan} to
projective morphisms between complex analytic spaces.
The overall strategy of the proofs follows the same lines as in the
algebraic case. A few remarks are in order. First, \cite{hacon-mckernan} was written
before \cite{bchm}, and the technically most demanding part of the
proof relies on the extension theorem established in
\cite{hacon-mckernan-b}. This extension theorem is so intricate that
even experts may find it difficult to remember all of its statements. In this paper, we avoid using the extension theorem and instead rely on
the minimal model program developed in \cite{fujino-minimal}.
Of course, the minimal model program itself ultimately depends, at
least in part, on extension theorems; thus, we are not entirely free
from them. Nevertheless, our approach, which only uses standard
arguments from the minimal model program, is arguably more
accessible to the reader. While we work in the setting of complex analytic spaces, in the
algebraic case, one can similarly obtain a proof of
\cite{hacon-mckernan} that superficially avoids the use of extension
theorems by applying \cite{bchm}. We emphasize once again that the existence of flips established
in \cite{hacon-mckernan-flips} still relies heavily on the
extension theorem. 
We hope that this paper makes Hacon--M\textsuperscript{c}Kernan's results on rational chain connectedness more accessible to readers. 

In Theorem \ref{p-thm5.1}, we extend \cite[Theorem 5.1]{hacon-mckernan} 
to the complex analytic setting. 
For convenience, we refer to \cite[Theorem 5.1]{hacon-mckernan} as 
{\em{Hacon--M\textsuperscript{c}Kernan's rational chain connectedness theorem}}; 
it constitutes the principal result of \cite{hacon-mckernan}. 
Since the precise formulation of Theorem \ref{p-thm5.1} is rather 
technical and lengthy, we refrain from reproducing it here. 
As an immediate consequence, we obtain the following theorem.

\begin{thm}[{\cite[Theorem~1.2]{hacon-mckernan}}]\label{p-thm1.1}
Let $X$ be a normal complex variety and let $\Delta$ be an effective
$\mathbb R$-divisor on $X$ such that $K_X+\Delta$ is $\mathbb R$-Cartier.
Let $f \colon X \to S$ be a projective morphism between complex analytic spaces.
Assume that $-K_X$ is $f$-big and that $-(K_X+\Delta)$
is $f$-semiample. 
Let $g \colon Y \to X$ be any bimeromorphic morphism, and let
$\pi \colon Y \to S$ be the induced morphism.
Then the connected components of every fiber of $\pi$ are rationally
chain connected modulo the inverse image of the non-klt locus of $(X,\Delta)$.
\end{thm}

Theorem \ref{p-thm1.1} admits several useful corollaries.

\begin{cor}[{\cite[Corollary~1.5]{hacon-mckernan}}]\label{p-cor1.2}
Let $(X,\Delta)$ be a divisorial log terminal pair.
Then we have:
\begin{itemize}
\item[(i)]
The fibers of any bimeromorphic morphism
$g \colon Y \to X$ are rationally chain connected.
\item[(ii)]
Assume further that $X$ is projective.
Then $X$ is rationally chain connected if and only if it is
rationally connected.
\end{itemize}
\end{cor}

\begin{cor}[{\cite[Corollary 1.6]{hacon-mckernan}}]\label{p-cor1.3}
Let $f \colon X \dashrightarrow Y$ be a meromorphic map
between normal complex varieties, and assume that $(X,\Delta)$
is a divisorial log terminal pair for some effective
$\mathbb R$-divisor $\Delta$.
Let $\Gamma \subset X \times Y$ be the graph of $f$,
and denote by $p \colon \Gamma \to X$ and $q \colon \Gamma \to Y$
the natural projections.
Then for every point $x \in X$ at which $f$ is not defined, 
\[
q\bigl(p^{-1}(x)\bigr) \subset Y
\]
is rationally chain connected. 
\end{cor}

\begin{cor}\label{p-cor1.4}
Let $f \colon X \dashrightarrow Y$ be a meromorphic map
between complex analytic spaces, and assume that $(X,\Delta)$
is divisorial log terminal.
If $Y$ contains no rational curves, then $f$ is a morphism.
\end{cor}

Corollaries~\ref{p-cor1.2}, \ref{p-cor1.3}, and \ref{p-cor1.4} 
are stated for divisorial log terminal pairs. 
However, we note the following simple observation.

\begin{rem}\label{p-rem1.5} 
Let $(X, \Delta)$ be a kawamata log terminal pair.
If $(X, \Delta)$ is in the algebraic setting, then $(X, \Delta)$ is
divisorial log terminal.
If $(X, \Delta)$ is in the complex analytic setting, then
$(U, \Delta|_U)$ is divisorial log terminal for any relatively compact
open subset $U$ of $X$.
For the details of the definition of divisorial log terminal pairs in the
complex analytic setting, see \cite[Section 3]{fujino-minimal}.
\end{rem}

In \cite{hacon-mckernan}, the following theorem is derived as an 
application of Hacon--M\textsuperscript{c}Kernan's rational chain 
connectedness theorem (see \cite[Theorem 5.1]{hacon-mckernan}). 
It is by now standard that Theorem \ref{p-thm1.6} admits a more 
direct proof.

\begin{thm}\label{p-thm1.6}
Let $X$ be a normal projective variety and let $\Delta$ be an effective 
$\mathbb R$-divisor on $X$ such that $K_X+\Delta$ is $\mathbb R$-Cartier. Assume that $-(K_X+\Delta)$ is ample. 
Then $X$ is rationally chain connected modulo $\Nklt(X, \Delta)$. 
Moreover, in the case where $(X, \Delta)$ is kawamata log terminal, 
$X$ is rationally connected. 
\end{thm}

Theorem \ref{p-thm1.7} is an application of Theorem \ref{p-thm1.6} and 
the theory of quasi-log structures. 

\begin{thm}[{Rational chain connectedness, see 
\cite[Theorem 1.14]{fujino-hyp} and \cite[Theorem 9.8]{fujino-quasi-log}}]\label{p-thm1.7} 
Let $\pi\colon X\to S$ be a projective morphism of 
complex analytic spaces with $\pi_*\mathcal O_X\simeq \mathcal O_S$ and let $[X, \omega]$ be a quasi-log complex analytic space. 
Assume that $-\omega$ is $\pi$-ample. Then $\pi^{-1}(P)$ is rationally chain connected modulo $\pi^{-1}(P)\cap X_{-\infty}$ for every point $P\in S$. In particular, if $\pi^{-1}(P)\cap X_{-\infty}=\emptyset$ further holds, that is, $[X, \omega]$ is quasi-log canonical in a neighborhood of $\pi^{-1}(P)$, then $\pi^{-1}(P)$ is rationally chain connected.
\end{thm}

The following corollary is a special case of Theorem \ref{p-thm1.7}. 
Consequently, Corollary \ref{p-cor1.8} follows from 
Theorem \ref{p-thm1.6} and is independent of 
Theorem \ref{p-thm1.1} and Theorem \ref{p-thm5.1}.

\begin{cor}[{\cite[Corollary~1.3]{hacon-mckernan}}]\label{p-cor1.8}
Let $X$ be a normal complex variety and let $\Delta$ be an effective
$\mathbb R$-divisor on $X$ such that $K_X+\Delta$ is $\mathbb R$-Cartier.
Let $f \colon X \to S$ be a projective morphism between normal complex varieties.
Assume that:
\begin{itemize}
\item[(1)]
$(X,\Delta)$ is kawamata log terminal, $-(K_X+\Delta)$ is $f$-nef,
and $-K_X$ is $f$-big; or
\item[(2)]
$(X,\Delta)$ is log canonical and $-(K_X+\Delta)$ is $f$-ample.
\end{itemize}
Then every connected component of every fiber of $f$ is
rationally chain connected.
\end{cor}

\begin{ack}
This work was partially supported by JSPS KAKENHI Grant Numbers JP21H04994 and JP23K20787. The author would like to express his sincere gratitude to Yoshinori Gongyo, Chuanjing Zhang, Shiyu Zhang, and Xi Zhang for their valuable comments and suggestions. The author would also like to thank Stefano Filipazzi and Nikolaos Tsakanikas for pointing out an error in Lemma~\ref{p-lem4.2} and for their helpful comments. Finally, the author would like to thank the anonymous referee for many valuable comments and suggestions, which greatly improved the presentation of this paper.
\end{ack}

We follow \cite{fujino-fundamental} and \cite{fujino-foundations} for the basic definitions 
and notation of the minimal model program. 
For the complex analytic setting, 
we refer to \cite{fujino-minimal} and \cite{fujino-cone}. 
For the basic properties of uniruled, rationally connected, and 
rationally chain connected varieties, we follow 
\cite{debarre} and \cite{kollar}. 
Throughout this paper, all algebraic varieties are defined over $\mathbb{C}$. 

\section{Preliminaries}\label{p-sec2}

We follow \cite[Chapter 2]{fujino-foundations}, \cite[Sections 2 and 3]{fujino-minimal}, and 
\cite[Chapter 2]{fujino-cone} for the standard definitions and notation 
in the minimal model program.

We begin by recalling several auxiliary definitions and basic properties 
that will be used throughout the paper.

\begin{defn}[Boundary part]\label{p-def2.1}
Let 
$
D=\sum_i d_i D_i
$ 
be an effective $\mathbb{R}$-divisor, where $d_i \in \mathbb{R}$,
each $D_i$ is a prime divisor, and $D_i \neq D_j$ for $i \neq j$.
Define
\[
D^b := \sum_i \min\{d_i,1\} D_i,
\qquad
D^{nb} := D - D^{b}.
\]
Then $D^{nb} \ge 0$, and we call $D^b$ the \emph{boundary part} of $D$.
\end{defn}

\begin{defn}[Non-klt locus]\label{p-def2.2}
Let $X$ be a normal complex variety, and let $\Delta$ be an 
effective $\mathbb{R}$-divisor on $X$ such that $K_X+\Delta$ is $\mathbb{R}$-Cartier. 
The \emph{non-klt locus} of $(X, \Delta)$, denoted by 
$\Nklt(X, \Delta)$, is the smallest closed analytic subset 
$Z \subset X$ such that $(X, \Delta)$ is kawamata log terminal on 
$X \setminus Z$. 

In this paper, we regard $\Nklt(X, \Delta)$ simply as a subset of $X$ 
and do not consider its complex analytic structure.
\end{defn}

\begin{defn}\label{p-def2.3}
Let $f\colon X \dashrightarrow Z$ be a dominant rational map
between projective varieties, and let $V$ be a Zariski closed subset of $X$.
We say that $V$ \emph{dominates} $Z$ if there exists an elimination of
indeterminacy $g\colon Y \to X$ of $f$, that is, a birational morphism
from a projective variety $Y$ such that the induced map
\[
h := f \circ g \colon Y \to Z
\]
is a morphism and satisfies $h(g^{-1}(V)) = Z$.
This notion is independent of the choice of the elimination of 
indeterminacy $g$.
\end{defn}

\begin{ex}\label{p-ex2.4} 
Let $\pi\colon \mathbb P^n\dashrightarrow \mathbb P^{n-1}$ be the linear projection from a point $P\in \mathbb P^n$.
The indeterminacy of $\pi$ is resolved by the blow-up of $\mathbb P^n$ at $P$, and the exceptional divisor is mapped isomorphically onto $\mathbb P^{n-1}$.
Hence $P$ dominates $\mathbb P^{n-1}$ in the sense of Definition~\ref{p-def2.3}.
\end{ex}

In the complex analytic setting, line bundles are not always represented by
Cartier divisors. Therefore, we define nefness for objects consisting of
$\mathbb{R}$-line bundles and $\mathbb{R}$-Cartier divisors. 

\begin{defn}[Nefness]\label{p-def2.5}
Let \( f \colon X \to Y \) be a projective morphism 
of complex analytic spaces, and let \( W \subset Y \) be a subset. 
Let \( \mathcal{L} \) be an \( \mathbb{R} \)-line bundle on \( X \), 
or the sum of an \( \mathbb{R} \)-line bundle and an 
\( \mathbb{R} \)-Cartier divisor. 
We say that \( \mathcal{L} \) is \emph{\( f \)-nef over \( W \)} 
if
\(
\mathcal{L} \cdot C \ge 0
\)
for every curve \( C \subset X \) such that \( f(C)\) is a point of $W$.
If \( \mathcal{L} \) is \( f \)-nef over \( Y \), 
we simply say that it is \( f \)-nef.
\end{defn}

In this paper, we adopt the following definition of bigness.
This is a complex analytic analogue of Hacon--M\textsuperscript{c}Kernan's definition of bigness.

\begin{defn}[Bigness]\label{p-def2.6}
Let $f\colon X\to S$ be a projective morphism of complex analytic spaces
such that $X$ is a normal complex variety.
Let $D$ be an $\mathbb R$-divisor on $X$.
If $S$ is Stein, then $D$ is said to be {\em big over $S$} or {\em $f$-big} if
\[
D\sim_{\mathbb Q,f} A+B,
\]
where $A$ is an $f$-ample $\mathbb Q$-divisor and $B$ is an effective
$\mathbb R$-divisor.
For a general base $S$, if $D|_{\pi^{-1}(U)}$ is big over $U$ for every
Stein open subset $U$ of $S$, then $D$ is said to be {\em big over $S$}
or {\em $f$-big}.
We note that $D$ is not assumed to be $\mathbb R$-Cartier.
\end{defn}

For the sake of completeness, we explicitly state the basepoint-free theorem in the complex analytic setting that will be used in this paper.

\begin{lem}[{Basepoint-free theorem, see \cite[Lemma 7.1]{hacon-mckernan}, \cite[Theorem 8.1]{fujino-minimal}, and \cite[Theorem 5.3.1]{fujino-cone}}]\label{p-lem2.7}
Let $(X, \Delta)$ be a kawamata log terminal pair, and let 
$f\colon X \to S$ be a projective morphism of complex analytic spaces. 
Suppose that $-(K_X + \Delta)$ is $f$-nef over a point $s \in S$, 
and that $-K_X$ is $f$-big. 
Then $-(K_X + \Delta)$ is $f$-semiample over some open neighborhood of $s$ in $S$.
\end{lem}

\begin{proof}[Proof of Lemma \ref{p-lem2.7}]
We may shrink $S$ around $s$ without further mention. 
Since $-K_X$ is $f$-big, we can write 
\[
-K_X \sim_{\mathbb{Q}, f} A + B,
\] 
where $A$ is an $f$-ample $\mathbb{Q}$-Cartier $\mathbb{Q}$-divisor 
and $B$ is an effective $\mathbb{Q}$-divisor. We set  
\[
\Theta := (1-\varepsilon)\Delta + \varepsilon B.
\] 
For sufficiently small $\varepsilon > 0$, the pair $(X, \Theta)$ remains kawamata log terminal.  
Moreover, we have 
\[
-(K_X + \Theta)  \sim_{\mathbb{Q}, f} -(1-\varepsilon)(K_X + \Delta) + \varepsilon A,
\] 
which is $f$-ample. 
Hence, by the basepoint-free theorem for 
$\mathbb R$-divisors in the complex analytic setting (see \cite[Theorem 8.1]{fujino-minimal}, and \cite[Theorem 5.3.1]{fujino-cone}), 
$-(K_X + \Delta)$ is $f$-semiample over some open neighborhood of $s \in S$, 
as desired.
\end{proof}

We next recall the notions of uniruledness, rational connectedness, 
and rational chain connectedness, which are central to the 
statements and arguments of this paper.

\begin{defn}[{Uniruledness, see \cite[Chapter~IV, 1.1 Definition]{kollar}}]\label{p-def2.8}
Let $X$ be an algebraic variety.
We say that $X$ is {\em uniruled} if there exist an algebraic 
variety $Y$ of dimension $\dim X - 1$ and a dominant rational map
\[
\mathbb{P}^1 \times Y \dashrightarrow X.
\]
\end{defn}

\begin{defn}[{Rational connectedness, see \cite[Chapter~IV, 3.6 
Proposition]{kollar}}]\label{p-def2.9}
Let $X$ be a compact complex variety. 
We say that $X$ is {\em rationally connected} if, for two general 
points $x_1, x_2 \in X$, there exists an irreducible rational curve
$C \subset X$ containing both $x_1$ and $x_2$.
\end{defn}

\begin{defn}[Rational chain connectedness]\label{p-def2.10}
Let $X$ be a compact complex analytic space.
We say that $X$ is {\em rationally chain connected} if, for any two points
$x_1, x_2 \in X$, there exists a connected curve $C \subset X$ such that
$x_1, x_2 \in C$ and every irreducible component of $C$ is rational.
\end{defn}

The following theorem is due to Koll\'ar, Miyaoka, and Mori.

\begin{thm}[{see \cite[Chapter~IV, 3.10 Theorem]{kollar}}]\label{p-thm2.11}
Let $X$ be a smooth projective variety.
Then $X$ is rationally chain connected if and only if $X$ is
rationally connected.
\end{thm}

We record an elementary lemma.

\begin{lem}\label{p-lem2.12}
Let $g\colon Z \to Y$ and $f\colon Y \to X$ be 
proper surjective morphisms of complex analytic spaces. 
Let $P$ be a point of $X$. 
If $(f\circ g)^{-1}(P)$ is rationally chain connected, then 
$f^{-1}(P)$ is also rationally chain connected. 
\end{lem}

\begin{proof}[Proof of Lemma~\ref{p-lem2.12}]
This follows since the natural morphism 
\[
(f\circ g)^{-1}(P) \to f^{-1}(P)
\]
is surjective.
\end{proof}

\begin{defn}[{\cite[Definition~1.1]{hacon-mckernan}}]\label{p-def2.13}
Let $X$ be a compact complex analytic space, and let
$V \subset X$ be a closed analytic subset. 
We say that $X$ is {\em rationally chain connected modulo $V$} if
\begin{itemize}
\item[(1)]
$V = \emptyset$ and $X$ is rationally chain connected, or
\item[(2)]
$V \neq \emptyset$ and for every $x \in X$, there exist a connected
pointed curve $C$ with marked points $0, \infty \in C$, such that
every irreducible component of $C$ is rational, and a morphism
$h_x \colon C \to X$ satisfying $h_x(0) = x$ and $h_x(\infty) \in V$.
\end{itemize}
\end{defn}

We prove a useful lemma for later use.

\begin{lem}\label{p-lem2.14}
Let $X$ be a smooth projective variety and let $V \subset X$ be a Zariski closed subset.
Let $f \colon Y \to X$ be a composite of blow-ups along smooth centers.
Then $X$ is rationally chain connected modulo $V$ if and only if $Y$ is rationally chain connected modulo $f^{-1}(V)$.
\end{lem}

\begin{proof}
The ``if'' direction is obvious from the definition.
We prove the ``only if'' direction.
By induction on the number of blow-ups, we may assume that
$f \colon Y \to X$ is the blow-up of $X$ along a smooth center
$Z \subset X$.
Then the exceptional locus $\Exc(f)$ is a projective space bundle over
$Z$.
Let $y \in Y$ and put $x=f(y)$.
Since $X$ is rationally chain connected modulo $V$, there exist a connected pointed curve $C$ with marked points $0,\infty\in C$, whose irreducible components are rational, and a morphism
$h\colon C\to X$ such that $h(0)=x$ and $h(\infty)\in V$. 
By replacing $C$ with its image $h(C)$, we may assume that
$h\colon C\to X$ is a closed embedding.
Since $\Exc(f)\to Z$ is a projective space bundle,
we can find a connected pointed curve $\widetilde C$
with marked points $0,\infty\in \widetilde C$,
whose irreducible components are rational,
and a closed embedding
$\widetilde h\colon \widetilde C\to Y$
such that
$f\circ\widetilde h=h\circ(f|_{\widetilde C})$,
$\widetilde h(0)=y$, and
$\widetilde h(\infty)\in f^{-1}(V)$.
\[
\xymatrix{
\widetilde C\ar[d]_{f|_{\widetilde C}}\ar@{^{(}->}[r]^-{\widetilde h}
&Y\ar[d]^-f\\
C\ar@{^{(}->}[r]_-h
&X
}
\]
Thus $Y$ is rationally chain connected modulo $f^{-1}(V)$.
\end{proof}

We recall two standard results that will be used later. 
Since they are well known, we state them without proof.

\begin{thm}[{\cite[Corollary~0.3]{bdpp}}]\label{p-thm2.15}
Let $Z$ be a smooth projective variety.
Then $Z$ is uniruled if and only if its canonical divisor $K_Z$ is not
pseudo-effective.
\end{thm}

\begin{thm}[{\cite[Corollary~1.4]{ghs}}]\label{p-thm2.16}
Let $X$ be a smooth projective variety, and let
$\phi\colon X \dashrightarrow Z$
be its maximal rationally connected fibration {\em{(}}{\em{MRC fibration}}, for short{\em{)}}.
Then $Z$ is not uniruled.
\end{thm}

We refer the reader to 
\cite[Chapter IV.5, Maximal Rationally Connected Fibrations]{kollar} 
for a comprehensive treatment of MRC fibrations.

\section{Inequality for the Kodaira Dimension}\label{p-sec3}

In this section, we discuss an inequality for the Kodaira dimension that is less widely known and may at first appear technical. For this reason, we present it here in detail. For background on the Kodaira dimension and its basic properties, we refer the reader to \cite{iitaka-conjecture}.

\begin{thm}[{\cite[Lemma~2.9 and Corollary~2.11]{hacon-mckernan-b}}]\label{p-thm3.1}
Let $(X, \Delta)$ be a projective log canonical pair such that 
$\Delta$ is a $\mathbb Q$-divisor. 
Let $\pi\colon X \to Y$ be a surjective morphism onto a smooth projective
variety $Y$ with connected fibers. 
Assume that 
\[
\kappa\bigl(X_y, (K_X + \Delta)|_{X_y}\bigr) \ge 0,
\]
where $X_y$ is a sufficiently general fiber of $\pi$.
Let $H$ be any ample Cartier divisor on $Y$ and let $\varepsilon > 0$
be any rational number.
Then
\begin{equation}\label{p-eq3.1}
\kappa\bigl(X, K_{X/Y} + \Delta + \varepsilon \pi^*H\bigr)
\ge \dim Y.
\end{equation}
In particular, if $K_Y$ is pseudo-effective, then
\begin{equation}\label{p-eq3.2}
\kappa\bigl(X, K_X + \Delta + \varepsilon \pi^*H\bigr)
\ge \dim Y.
\end{equation}
\end{thm}

For the sake of completeness, we give a proof of
Theorem~\ref{p-thm3.1}.

\begin{proof}[Proof of Theorem~\ref{p-thm3.1}]
We divide the proof into three steps.
In Step~\ref{p-step3.1.1}, we prove~\eqref{p-eq3.1} under the additional
assumption that $\pi$ is equidimensional.
In Step~\ref{p-step3.1.2}, we remove this assumption and establish
\eqref{p-eq3.1} in full generality.
Finally, in Step~\ref{p-step3.1.3}, we deduce~\eqref{p-eq3.2}.

\setcounter{step}{0}

\begin{step}\label{p-step3.1.1}
In this step, we prove~\eqref{p-eq3.1} under the extra assumption that
$\pi$ is equidimensional.

Let $a$ be a positive integer such that $a(K_X+\Delta)$ is Cartier and
\[
\pi_*\mathcal{O}_X\bigl(a(K_{X/Y}+\Delta)\bigr)\neq 0.
\]
Then, by \cite[Proposition~9.1]{fujino-weak}, we have
\[
\kappa\bigl(X, a(K_{X/Y}+\Delta)+\pi^*H\bigr)\ge \dim Y.
\] 
It follows from the equality
\[
K_{X/Y}+\Delta+\varepsilon \pi^*H
=
\frac{1}{a}\left(a(K_{X/Y}+\Delta)+\pi^*H\right)
+
\left(\varepsilon-\frac{1}{a}\right)\pi^*H
\]
and the fact that $\pi^*H$ is $\mathbb{Q}$-linearly equivalent to an effective $\mathbb{Q}$-divisor that
\[
\kappa\bigl(X, K_{X/Y}+\Delta+\varepsilon \pi^*H\bigr)\ge \dim Y
\]
for any rational number $\varepsilon>0$.
\end{step}

\begin{step}\label{p-step3.1.2}
In this step, we prove~\eqref{p-eq3.1} in full generality.

By \cite[Theorem 2.1]{ak} (see also
\cite[Theorem 1.1]{adk}), we can construct a commutative diagram
\[
\xymatrix{
X \ar[d]_-{\pi} & X' \ar[l]_-{h} \ar[d]^-{\pi'} \\
Y & Y' \ar[l]^-{g}
}
\]
with the following properties:
\begin{itemize}
\item[(1)]
$\pi'\colon X' \to Y'$ is an equidimensional surjective morphism from a
normal projective variety $X'$ to a smooth projective variety $Y'$,
\item[(2)]
$h$ and $g$ are birational morphisms,
\item[(3)]
$X'$ has only quotient singularities, and there exists a nonempty
Zariski open subset $U_{X'} \subset X'$ such that
$(U_{X'} \subset X')$ is toroidal,
\item[(4)]
$\Exc(h) \cup \Supp h^{-1}_*\Delta \subset X' \setminus U_{X'}$.
\end{itemize}

We write
\[
K_{X'}+\Delta' = h^*(K_X+\Delta)+E,
\]
where $\Delta'$ and $E$ are effective $\mathbb{Q}$-divisors that have no common irreducible components.
Since $(X,\Delta)$ is log canonical, the coefficients of $\Delta'$ are at most one and $E$ is $h$-exceptional.
By property~(4), we have $\Supp \Delta' \subset X' \setminus U_{X'}$, and hence $(X',\Delta')$ is log canonical.
We also write
\[
K_{Y'} = g^*K_Y + F,
\]
where $F$ is an effective $g$-exceptional divisor.
Then
\begin{equation}\label{p-eq3.3}
K_{X'/Y'}+\Delta'
= h^*(K_{X/Y}+\Delta)+E-\pi'^*F.
\end{equation}

Let $H'$ be an ample Cartier divisor on $Y'$, and let $\varepsilon'>0$
be a rational number such that
\begin{equation}\label{p-eq3.4}
\kappa\bigl(Y', g^*H-\varepsilon' H'\bigr)\ge 0.
\end{equation}
Then we obtain
\[
\begin{split}
\kappa\bigl(X, K_{X/Y}+\Delta+\varepsilon \pi^*H\bigr)
&\ge \kappa\bigl(X', K_{X'/Y'}+\Delta'+\varepsilon \pi'^*g^*H\bigr) \\
&\ge \kappa\bigl(X', K_{X'/Y'}+\Delta'
      +\varepsilon\varepsilon'\pi'^*H'\bigr) \\
&\ge \dim Y.
\end{split}
\] 
The first and second inequalities follow from \eqref{p-eq3.3} and \eqref{p-eq3.4}, respectively. The last inequality follows from Step~\ref{p-step3.1.1}, since $\pi'$ is equidimensional.
\end{step}

\begin{step}\label{p-step3.1.3}
In this step, we prove~\eqref{p-eq3.2}.

Assume that $K_Y$ is pseudo-effective.
Let $\varepsilon'$ be a rational number such that
$0<\varepsilon'<\varepsilon$. 
Since $K_Y$ is pseudo-effective,
$K_Y+(\varepsilon-\varepsilon')H$ is $\mathbb{Q}$-linearly equivalent to an effective
$\mathbb{Q}$-divisor.
Since
\[
K_X+\Delta+\varepsilon \pi^*H
= K_{X/Y}+\Delta+\varepsilon' \pi^*H
  + \pi^*\bigl(K_Y+(\varepsilon-\varepsilon')H\bigr),
\]
we have
\[
\kappa\bigl(X, K_X+\Delta+\varepsilon \pi^*H\bigr)
\ge
\kappa\bigl(X, K_{X/Y}+\Delta+\varepsilon' \pi^*H\bigr)
\ge \dim Y,
\]
where the last inequality follows from~\eqref{p-eq3.1}.
This proves~\eqref{p-eq3.2}.
\end{step}

This completes the proof of Theorem~\ref{p-thm3.1}.
\end{proof}

\begin{cor}\label{p-cor3.2}
In Theorem~\ref{p-thm3.1}, the assumption that $(X,\Delta)$ is log
canonical can be replaced by the following slightly weaker condition,
which is required in the proof of Proposition~\ref{p-prop4.1}.
Namely, $\Delta$ is an effective $\mathbb{Q}$-divisor on $X$ such that 
$K_X+\Delta$ is $\mathbb Q$-Cartier and
$(X,\Delta)$ is log canonical over the generic point of $Y$.
\end{cor}

\begin{proof}[Proof of Corollary~\ref{p-cor3.2}]
By taking a resolution of singularities, we may assume that $X$ is
smooth and that $\Supp \Delta$ is a simple normal crossing divisor on
$X$.
We can choose an effective $\mathbb{Q}$-divisor $\Delta'$ on $X$ such
that $(X,\Delta')$ is log canonical, $\Delta \ge \Delta'$, and
$\Delta=\Delta'$ holds over the generic point of $Y$.
Applying Theorem~\ref{p-thm3.1} to the pair $(X,\Delta')$ and using the
inequality $\Delta \ge \Delta'$, we obtain the desired inequalities.
\end{proof}

We note that Theorem~\ref{p-thm3.1}, and hence Corollary~\ref{p-cor3.2},
ultimately rely on Campana's twisted weak positivity
(see \cite[Theorem~1.1]{fujino-weak}).
Alternatively, one may apply directly the theory of mixed
$\omega$-sheaves developed in \cite{fujino-omega}.
In fact, Theorem~\ref{p-thm3.1} and Corollary~\ref{p-cor3.2} follow
immediately from Theorem~\ref{p-thm3.3} below.
In any case, Theorem~\ref{p-thm3.1} and Corollary~\ref{p-cor3.2} are now
well known to experts.

\begin{thm}[{\cite[Corollary~9.5]{fujino-omega}}]\label{p-thm3.3}
Let $\pi\colon X \to Y$ be a surjective morphism
from a normal projective variety $X$ onto a smooth projective variety
$Y$.
Let $\Delta$ be an effective $\mathbb{Q}$-divisor on $X$ such that
$K_X+\Delta$ is $\mathbb{Q}$-Cartier and $(X,\Delta)$ is log canonical
over the generic point of $Y$.
Let $H$ be an ample Cartier divisor on $Y$.

Then there exists a positive integer $m$ with the following property:
for every integer $a \ge 2$ such that $a(K_X+\Delta)$ is Cartier and
\[
\pi_*\mathcal{O}_X\bigl(a(K_X+\Delta)\bigr) \neq 0,
\]
the sheaf
\[
\mathcal{O}_Y(mH)\otimes
\pi_*\mathcal{O}_X\bigl(a(K_{X/Y}+\Delta)\bigr)
\]
is generically globally generated.
In particular, for every such $a$, we have
\[
\kappa\bigl(X, a(K_{X/Y}+\Delta) + m\pi^*H\bigr) \ge 0.
\]
\end{thm}

\section{A Criterion due to Hacon--M\textsuperscript{c}Kernan}\label{p-sec4}

In this section, for completeness, we recall the criterion 
of Hacon and M\textsuperscript{c}Kernan for rational chain connectedness 
modulo a subset. 
Our exposition follows \cite[Section~4]{hacon-mckernan}. 
Lemma \ref{p-lem4.6} is new and will play a crucial role 
from the perspective of minimal model theory.

\begin{prop}[{\cite[Proposition~4.1]{hacon-mckernan}}]\label{p-prop4.1}
Let $X$ be a normal projective variety and let $\Delta$ be an effective
$\mathbb Q$-divisor such that $K_X+\Delta$ is $\mathbb Q$-Cartier.
Let $h \colon X \to F$ be a dominant morphism and
$t \colon F \dashrightarrow Z$ a dominant rational map between projective varieties.
Assume that:
\begin{itemize}
\item[(1)]  
The non-klt locus of $(X,\Delta)$ does not dominate $Z$ via the induced map
$t \circ h \colon X \dashrightarrow Z$.
\item[(2)]  
$K_X+\Delta$ has nonnegative Kodaira dimension on the general fiber of
$X \dashrightarrow Z$, that is, if $g \colon Y \to X$ resolves the indeterminacy
of $X \dashrightarrow Z$ and $Y$ is smooth projective, then
$g^*(K_X+\Delta)$ has Kodaira dimension at least zero on the general fiber
of the induced morphism $Y \to Z$.
\item[(3)]  
$\kappa(X,K_X+\Delta)\le 0$.
\item[(4)]  
There exists an effective big $\mathbb Q$-Cartier $\mathbb Q$-divisor $A$ on $F$ satisfying
$h^*A \le \Delta$.
\end{itemize}
Then $Z$ is either a point or uniruled.
\end{prop}

\begin{proof}[Proof of Proposition~\ref{p-prop4.1}]
It suffices to show that $Z$ is a point under the assumption that $Z$ is
not uniruled.
By resolution of singularities, we may assume without loss of
generality that $Z$ is smooth.

Take a resolution of singularities $g\colon Y \to X$ such that the
induced rational map $Y \dashrightarrow Z$ is a morphism
$\psi\colon Y \to Z$, and such that
$E' := \Exc(g)$ and
$\Exc(g) \cup \Supp g^{-1}_*\Delta$ are simple normal crossing divisors
on $Y$. 
After replacing $F$ by a resolution of singularities
$\alpha \colon F'\to F$, we may further assume that $\psi\colon Y\to Z$ factors
through $F'$ as follows:
\[
\xymatrix{
Y \ar[d]_-g\ar[r]\ar@/^12pt/[rr]^\psi& F'\ar[r]\ar[d]_\alpha& Z\ar@{=}[d]\\
X \ar[r]_-h& F\ar@{-->}[r]& Z.
}
\]

We write
\[
K_Y + \Theta = g^*(K_X + \Delta) + E,
\]
where $\Theta$ and $E$ are effective $\mathbb{Q}$-divisors with no common
irreducible components, $g_*\Theta = \Delta$, and $E$ is
$g$-exceptional.
Set
\[
\Gamma := \Theta + \varepsilon E',
\]
where $\varepsilon>0$ is a sufficiently small rational number.

By condition~(1) and the above construction, the pair $(Y,\Gamma)$ is
kawamata log terminal over the generic point of $Z$.
Moreover, by condition~(2), we have
\[
\kappa\bigl(Y_z, (K_Y+\Gamma)|_{Y_z}\bigr) \ge 0,
\]
where $Y_z$ is a sufficiently general fiber of $\psi$.

By construction, the support of $\Gamma$ contains the exceptional locus
of $g$.
Therefore, by condition~(4), $\Gamma$ contains the pull-back of a big
$\mathbb{Q}$-Cartier $\mathbb{Q}$-divisor on $F$. 
Hence, after pulling back to $F'$, $\Gamma$ contains the pull-back of a
big $\mathbb{Q}$-Cartier $\mathbb{Q}$-divisor on $F'$. 
By Kodaira's lemma applied on $F'$, possibly after replacing $\Gamma$ by a
$\mathbb{Q}$-linearly equivalent $\mathbb{Q}$-divisor, we can find an
ample $\mathbb{Q}$-divisor $G$ on $Z$ such that
\[
\psi^*G \le \Gamma.
\]
Since the replacement can be chosen sufficiently small, we may assume that
$(Y,\Gamma)$ remains kawamata log terminal over the generic point of $Z$.

Since $Z$ is not uniruled by assumption, Theorem~\ref{p-thm2.15} implies
that $K_Z$ is pseudo-effective. 
Observe that $(Y,\Gamma-\psi^*G)$ is kawamata log terminal, and hence log
canonical, over the generic point of $Z$. Moreover,
\[
\kappa\bigl(Y_z,(K_Y+(\Gamma-\psi^*G))|_{Y_z}\bigr)
=
\kappa\bigl(Y_z,(K_Y+\Gamma)|_{Y_z}\bigr)
\ge 0,
\]
where $Y_z$ is a sufficiently general fiber of $\psi$.
Therefore, by Corollary~\ref{p-cor3.2}, we obtain
\[
\kappa\bigl(Y, K_Y+\Gamma\bigr)
=
\kappa\bigl(Y, K_Y+(\Gamma-\psi^*G)+\psi^*G\bigr)
\ge \dim Z.
\]

On the other hand,
\[
\kappa\bigl(Y, K_Y+\Gamma\bigr)
= \kappa\bigl(X, K_X+\Delta\bigr) \le 0
\]
by condition~(3).
It follows that $\dim Z = 0$, that is, $Z$ is a point.
This completes the proof of Proposition~\ref{p-prop4.1}.
\end{proof}

\begin{lem}[{\cite[Lemma~4.2]{hacon-mckernan}}]\label{p-lem4.2}
Let $F$ be a projective variety.
\begin{itemize}
\item[(i)]
$F$ is rationally connected if and only if every nonconstant
dominant rational map
$t \colon F \dashrightarrow Z$ of projective varieties has uniruled target.
\end{itemize}

Assume in addition that $F$ is smooth.
\begin{itemize}
\item[(ii)]
$F$ is rationally chain connected modulo $V$ if and only if for every
nonconstant dominant rational map
$t \colon F \dashrightarrow Z$ of projective varieties, the target $Z$
is either uniruled or dominated by $V$.
\end{itemize}
\end{lem}

\begin{proof}[Proof of Lemma~\ref{p-lem4.2}]
\setcounter{step}{0}
In Step~\ref{p-step4.2.1}, we prove~(i).
In Step~\ref{p-step4.2.2}, we prove~(ii).

\begin{step}\label{p-step4.2.1}
In this step, we prove~(i).

Assume first that $F$ is rationally connected, and let
$t\colon F \dashrightarrow Z$ be a nonconstant dominant rational map.
Then $Z$ is also rationally connected.
In particular, $Z$ is uniruled.
This proves the ``only if'' direction.

Conversely, assume that every nonconstant dominant rational map
$t\colon F \dashrightarrow Z$ has uniruled target.
In particular, since the identity map $F \dashrightarrow F$ is a
dominant rational map, it follows that $F$ itself is uniruled.

Let $F' \to F$ be a resolution of singularities from a smooth projective variety $F'$.
Since $F$ is uniruled, so is $F'$.
Let $F' \dashrightarrow Z$ be the maximal rationally connected fibration of $F'$.
By Theorem~\ref{p-thm2.16}, the base $Z$ is not uniruled.
Since $F' \to F$ is birational, it induces a dominant rational map
$t\colon F \dashrightarrow Z$.
Hence $Z$ must be a point.
Therefore, $F'$ is rationally connected, and consequently so is $F$.
This completes the proof of~(i).
\end{step}

\begin{step}\label{p-step4.2.2}
In this step, we prove~(ii).
By~(i) and Theorem~\ref{p-thm2.11}, we may assume throughout this step that
$V\neq\emptyset$. If $V=F$, then there is nothing to prove. 
Thus we may also assume that $V\ne F$.

Assume first that $F$ is rationally chain connected modulo $V$, and let
$t\colon F \dashrightarrow Z$ be a nonconstant dominant rational map. 
By Lemma~\ref{p-lem2.14}, after eliminating the indeterminacy of
$t\colon F\dashrightarrow Z$ by a finite sequence of blow-ups along smooth centers (see, for example, \cite[Corollary 3.18]{kollar-resolution}),
we may assume that $t\colon F\to Z$ is a morphism.
If $V$ does not dominate $Z$, then $Z$ is covered by images of rational curves.
In particular, $Z$ is uniruled.
This proves the ``only if'' direction.

Conversely, assume that for every nonconstant dominant rational map
$t\colon F \dashrightarrow Z$, its target $Z$ is either uniruled or dominated by $V$.
As in Step~\ref{p-step4.2.1}, by considering the identity map
$F \dashrightarrow F$, we see that $F$ is uniruled.

Let $F \dashrightarrow Z$ be the maximal rationally connected fibration of $F$.
By Theorem~\ref{p-thm2.16}, the base $Z$ is not uniruled.
Therefore, by assumption, $Z$ must be dominated by $V$.
This implies that $F$ is rationally chain connected modulo $V$.
This completes the proof of assertion~(ii).
\end{step}

This completes the proof of Lemma~\ref{p-lem4.2}.
\end{proof}

\begin{ex}\label{p-ex4.3}
Let $X$ be the cone over an elliptic curve $E$.
Then $X$ is rationally chain connected, and there exists a dominant rational map
$X\dashrightarrow E$.
In particular, $X$ is rationally chain connected modulo every Zariski closed subset of $X$.
On the other hand, $X$ is not rationally connected.
\end{ex}

\begin{rem}\label{p-rem4.4}
Lemma~4.2~(2) of \cite{hacon-mckernan} is stated for a normal variety $F$.
However, the statement is not correct as stated.
Indeed, the variety $X$ in Example~\ref{p-ex4.3} gives a counterexample to
\cite[Lemma~4.2~(2)]{hacon-mckernan}.
Therefore, in Lemma~\ref{p-lem4.2}~(ii), we assume that $F$ is smooth.
\end{rem}

\begin{cor}[{\cite[Corollary~4.3]{hacon-mckernan}}]\label{p-cor4.5}
Let $X$ be a normal projective variety and let $\Delta$ be an effective
$\mathbb Q$-divisor such that $K_X+\Delta$ is $\mathbb Q$-Cartier.
Let $h \colon X \to F$ be a dominant morphism of projective varieties 
such that $F$ is smooth.
Assume that for every dominant rational map
$t \colon F \dashrightarrow Z$ of projective varieties,
one of the following holds:
\begin{itemize}
\item the non-klt locus of $(X,\Delta)$ dominates $Z$, or
\item conditions {\em (2)--(4)} of Proposition~\ref{p-prop4.1} are satisfied.
\end{itemize}
Then $F$ is rationally chain connected modulo the image
$R \subset F$ of the non-klt locus of $(X,\Delta)$.
\end{cor}

\begin{proof}[Proof of Corollary~\ref{p-cor4.5}]
We apply the criterion given in Lemma~\ref{p-lem4.2}.
Let $t\colon F \dashrightarrow Z$ be an arbitrary dominant rational map of projective varieties.

If the image $R$ of the non-klt locus of $(X,\Delta)$ dominates $Z$,
there is nothing to prove.
Otherwise, $R$ does not dominate $Z$, and we claim that $Z$ is either a
point or uniruled.

Indeed, since $R$ does not dominate $Z$, condition~(1) of
Proposition~\ref{p-prop4.1} is satisfied.
By assumption, conditions~(2)--(4) of Proposition~\ref{p-prop4.1} also
hold.
Therefore, Proposition~\ref{p-prop4.1} implies that $Z$ is either
uniruled or a point.

By Lemma~\ref{p-lem4.2}, this shows that $F$ is rationally chain connected
modulo $R$.
This completes the proof of Corollary~\ref{p-cor4.5}.
\end{proof}

In this paper, we apply Corollary \ref{p-cor4.5} via the following lemma.

\begin{lem}\label{p-lem4.6}
Let $W$ be a normal projective variety, and let $\Delta_W$ be an effective
$\mathbb{Q}$-divisor on $W$ such that $K_W+\Delta_W$ is $\mathbb{Q}$-Cartier.
Let $\varphi\colon V \to W$ be a birational morphism from a smooth
projective variety $V$ such that both $\Exc(\varphi)$ and
$
\Exc(\varphi)\cup \Supp \varphi^{-1}_*\Delta_W
$
are simple normal crossing divisors on $V$.
Write
\[
K_V+\Delta_V=\varphi^*(K_W+\Delta_W)+E,
\]
where $\Delta_V$ and $E$ have no common irreducible components,
$\varphi_*\Delta_V=\Delta_W$, and $E$ is $\varphi$-exceptional.

Assume the following:
\begin{itemize}
\item[(a)] $K_W+\Delta_W \sim_{\mathbb{Q}} P$ for some effective
$\mathbb{Q}$-divisor $P$, equivalently, $\kappa(W,K_W+\Delta_W)\ge 0$. 
\item[(b)] $\kappa(W, K_W+\Delta_W^b)\le 0$, 
where $\Delta^b_W$ denotes the boundary part of $\Delta_W$.
\item[(c)] There exists an effective big $\mathbb{Q}$-Cartier
$\mathbb{Q}$-divisor $Q$ on $W$ such that $Q \le \Delta^b_W$.
\end{itemize}
Then $V$ is rationally chain connected modulo
$\Nklt(V,\Delta_V)$.
\end{lem}

\begin{proof}[Proof of Lemma \ref{p-lem4.6}]We first recall the notation introduced in
Definition~\ref{p-def2.1}. For any effective
$\mathbb R$-divisor $D$, we write
$D^b$ for the boundary part of $D$ and
$D^{nb}:=D-D^b$. 

Set
\[
E^\dagger := \Exc(\varphi)
\quad \text{and} \quad
\Theta := \Delta_V + \varepsilon E^\dagger,
\]
where $\varepsilon > 0$ is a sufficiently small rational number.
Then
\[
\Nklt(V,\Theta^b)
= \Nklt(V,\Theta)
= \Nklt(V,\Delta_V).
\]

We apply Corollary \ref{p-cor4.5} to the pair $(V,\Theta^b)$
with respect to the identity morphism $V \to V$.
Let $t\colon V \dashrightarrow Z$ be a dominant rational map of projective varieties.
If $\Nklt(V,\Theta^b)$ dominates $Z$, there is nothing to prove.
Hence we may assume that $\Nklt(V,\Theta^b)$ does not dominate $Z$. 
It suffices to verify conditions (2)--(4) in
Proposition \ref{p-prop4.1}. 
Since $\Exc(\varphi) \subset \Supp \Theta^b$ by construction,
condition (4) follows from condition (c). Next, we compute
\begin{equation}\label{p-eq4.1}
\begin{split}
K_V+\Theta^b
&= \varphi^*(K_W+\Delta_W)
   - \Theta^{nb}
   + E
   + \varepsilon E^\dagger \\
&= \varphi^*(K_W+\Delta_W^b)
   + \varphi^*\Delta_W^{nb}
   - \Theta^{nb}
   + E
   + \varepsilon E^\dagger.
\end{split}
\end{equation}
Note that
\[
\varphi^*\Delta_W^{nb}
- \Theta^{nb}
+ E
+ \varepsilon E^\dagger
\]
is $\varphi$-exceptional.
Therefore, by (b) and \eqref{p-eq4.1},
\[
\kappa(V,K_V+\Theta^b)
\le
\kappa(W,K_W+\Delta_W^b)
\le 0.
\]
Thus condition (3) holds. Finally, by \eqref{p-eq4.1}, condition (a),
the effectivity of $E+\varepsilon E^\dagger$,
and the inclusion
\[
\Supp \Theta^{nb}
\subset \Nklt(V,\Theta^b),
\]
we see that condition (2) is satisfied.

Hence Corollary \ref{p-cor4.5} applies, and we conclude that
$V$ is rationally chain connected modulo
\[
\Nklt(V,\Theta^b)
=
\Nklt(V,\Delta_V).
\]
This completes the proof.
\end{proof}

As a straightforward application of Lemma \ref{p-lem4.6}, 
we establish the rational connectedness of divisorial log terminal Fano pairs.

\begin{cor}\label{p-cor4.7}
Let $(X,\Delta)$ be a projective divisorial log terminal pair 
such that $-(K_X+\Delta)$ is ample. 
Then $X$ is rationally connected. 
\end{cor}

\begin{proof}[Proof of Corollary \ref{p-cor4.7}]
By slightly perturbing $\Delta$, we may assume that 
$(X,\Delta)$ is kawamata log terminal and 
$-(K_X+\Delta)$ is an ample $\mathbb Q$-divisor. 
Since $-(K_X+\Delta)$ is ample, we take an effective 
$\mathbb{Q}$-divisor $D$ such that 
\[
D \sim_{\mathbb{Q}} -(K_X+\Delta)
\]
and $(X,\Delta+D)$ is kawamata log terminal. 
Then
\[
K_X+\Delta+D \sim_{\mathbb{Q}} 0.
\]
By Lemma \ref{p-lem4.6}, 
there exists a resolution $\varphi\colon V \to X$ 
such that $V$ is rationally chain connected. 
Since $V$ is smooth, it follows that $V$ is rationally connected 
(see Theorem \ref{p-thm2.11}). 
Hence $X$ is rationally connected.
\end{proof}

We also need the following variant of Lemma~\ref{p-lem4.6} 
for the proof of Hacon--M\textsuperscript{c}Kernan's rational 
chain connectedness theorem (see Theorem~\ref{p-thm5.1}).

\begin{lem}\label{p-lem4.8} 
Let $X$ be a normal projective variety and let $\Delta$ be an 
effective $\mathbb Q$-divisor on $X$ such that $K_X+\Delta$ is 
$\mathbb Q$-Cartier.
Let $\varphi\colon X\to W$ be a projective surjective morphism
onto a normal projective variety $W$ with connected fibers
such that $-(K_X+\Delta)$ is $\varphi$-ample and
$\varphi(\Nklt(X, \Delta))=W$.
Let $\psi\colon V\to X$ be a birational morphism
from a smooth projective variety $V$ such that
both $\Exc(\psi)$ and
$
\Exc(\psi)\cup \Supp \psi^{-1}_*\Delta
$
are simple normal crossing divisors on $V$.
Write
\[
K_V+\Delta_V=\psi^*(K_X+\Delta)+E,
\]
where $\Delta_V$ and $E$ have no common irreducible components,
$\psi_*\Delta_V=\Delta$,
and $E$ is $\psi$-exceptional.
Then $V$ is rationally chain connected modulo
$\Nklt(V, \Delta_V)$.
\end{lem}

\begin{proof}[Proof of Lemma~\ref{p-lem4.8}]
Let $H$ be a sufficiently ample Cartier divisor on $W$
such that $-(K_X+\Delta)+\varphi^*H$ is ample.
Choose a general effective $\mathbb Q$-divisor $A$ on $X$
with sufficiently small coefficients such that
\[
A\sim_{\mathbb Q} -(K_X+\Delta)+\varphi^*H.
\]
Set $\Delta_X:=\Delta+A$.
Then
\[
K_X+\Delta_X \sim_{\mathbb Q} \varphi^*H,
\]
and hence $\kappa(X,K_X+\Delta_X)\ge 0$. 
Moreover, since the coefficients of $A$ are sufficiently small, $A\leq \Delta^b_X$. 

Pulling back to $V$, we obtain
\[
K_V+\Delta_V+\psi^{-1}_*A
=
\psi^*(K_X+\Delta_X)+E.
\] 
Observe that $\psi^*A=\psi^{-1}_*A$, since $A$ is general.
Since $A$ is general, we may further assume that
\[
\Exc(\psi)\cup
\Supp \psi^{-1}_*A\cup
\Supp \psi^{-1}_*\Delta
\]
is a simple normal crossing divisor on $V$.

We now apply Lemma~\ref{p-lem4.6} to the general fiber of
$\varphi\colon X\to W$ with respect to the pair $(X,\Delta_X)$. 
Let $X_w$ be a general fiber of $\varphi\colon X\to W$.
Then
\[
K_{X_w}+\Delta_{X_w}
:=
(K_X+\Delta_X)|_{X_w}
\sim_{\mathbb Q}0. 
\]
Moreover, by construction, $\Delta^b_{X_w}\ge A|_{X_w}$. 
Therefore, $(X_w,\Delta_{X_w})$ satisfies conditions~(a)--(c) of
Lemma~\ref{p-lem4.6}.
Applying Lemma~\ref{p-lem4.6} to
\[
V_w:=(\varphi\circ\psi)^{-1}(w)\longrightarrow X_w,
\]
we obtain that $V_w$ is rationally chain connected modulo
$\Nklt\bigl(V_w,\Delta_{V_w}+\psi^{-1}_*A|_{V_w}\bigr)$, 
where
$
K_{V_w}+\Delta_{V_w}
:=
(K_V+\Delta_V)|_{V_w}$. 
Since
\[
\Nklt\bigl(V_w,\Delta_{V_w}+\psi^{-1}_*A|_{V_w}\bigr)
=
\Nklt\bigl(V,\Delta_V+\psi^{-1}_*A\bigr)|_{V_w}
=
\Nklt\bigl(V,\Delta_V\bigr)|_{V_w},
\]
it follows that $V$ is rationally chain connected modulo
$\Nklt(V,\Delta_V)$.
This completes the proof.
\end{proof}

\section{The Hacon--M\textsuperscript{c}Kernan Rational Chain Connectedness Theorem}\label{p-sec5}

In this section, we establish the Hacon--M\textsuperscript{c}Kernan rational chain connectedness theorem in the complex analytic setting:~Theorem \ref{p-thm5.1}.

\begin{thm}[{\cite[Theorem 5.1]{hacon-mckernan}}]\label{p-thm5.1}
Let $f\colon X \to S$ be a projective morphism between normal complex
varieties such that $f_*\mathcal{O}_X \simeq \mathcal{O}_S$.
Let $\Delta$ be an effective $\mathbb{Q}$-divisor on $X$ such that
$K_X + \Delta$ is $\mathbb{Q}$-Cartier, $\Delta$ is $f$-big, and
$K_X + \Delta \sim_{\mathbb{Q}, f} 0$.
Fix a point $s \in S$. 
After possibly shrinking $S$ around $s$, we may construct
a resolution $g\colon Y \to X$ and effective $\mathbb{Q}$-divisors
$\Gamma$, $E$, and $G$ satisfying the following properties. 
\begin{itemize}

\item[(1)]
Let $g\colon Y \to X$ be a resolution of singularities such that
$\pi := f \circ g\colon Y \to S$ is projective,
$\pi^{-1}(s)$ and $\Exc(g)$ are simple normal crossing divisors on $Y$, and
\[
\Exc(g) \cup \Supp g^{-1}_*\Delta \cup \pi^{-1}(s)
\]
is a simple normal crossing divisor on $Y$.
Given any bimeromorphic morphism $X' \to X$, we may assume further that
$g\colon Y \to X$ factors through $X' \to X$.

\item[(2)]
We have
\[
K_Y + \Gamma \sim_{\mathbb{Q}, \pi} E,
\]
where $\Gamma$ and $E$ have no common irreducible components,
$\Supp(\Gamma + E)$ is a simple normal crossing divisor on $Y$,
$E$ is $g$-exceptional, and $\Gamma$ can be written as $\Gamma = A + B$
with $A$ an effective general $\pi$-ample $\mathbb{Q}$-divisor
with small coefficients and $B$ an effective $\mathbb{Q}$-divisor.

\item[(3)]
We have $\Nklt(Y, \Gamma) \subset g^{-1}\Nklt(X, \Delta)$.
In particular, if $(X, \Delta)$ is kawamata log terminal, then
$(Y, \Gamma)$ is kawamata log terminal.

\item[(4)]
There exists an effective $\mathbb{Q}$-Cartier $\mathbb{Q}$-divisor $G$ on $S$ such that
\[
\Supp(\Gamma+E)\cup\Supp(\pi^*G)\cup\pi^{-1}(s)
\]
is a simple normal crossing divisor on $Y$.
For $t \in [0,1]$, set $\Delta_t := \Delta + t f^*G$. 
Let $\Gamma_t$ and $E_t$ be the effective $\mathbb{Q}$-divisors obtained
from $\Gamma + t\pi^*G$ and $E$, respectively, by subtracting their common
irreducible components, so that
\[
K_Y + \Gamma_t \sim_{\mathbb{Q}, \pi} E_t.
\]
Let $V_t$ be the closure of
\[
\Nklt(Y, \Gamma_t) \setminus \Nklt(Y, \Gamma).
\]
Let $F$ be the union of the irreducible components of $\pi^{-1}(s)$
whose discrepancies with respect to $K_X+\Delta$ are greater than $-1$.
Then $F = V_1$ holds.

\item[(5)]
Moreover, there exist rational numbers
\[
0 = t_0 < t_1 < \cdots < t_k \leq 1
\]
such that
\[
V_{t_i} = F_1 \cup \cdots \cup F_i,
\]
where $F_1, \dots, F_k$ are the irreducible components of $F$.
For any $0 < \varepsilon \ll 1$, we have
\[
V_{t_i - \varepsilon} = V_{t_{i-1}}.
\]

\item[(6)]
Each $F_i$ is rationally chain connected modulo
\[
W_i := F_i \cap \Nklt(Y, \Gamma_{t_{i-1}}).
\]
Therefore, $F$ is rationally chain connected modulo $\Nklt(Y,\Gamma)$,
and hence modulo $g^{-1}\Nklt(X,\Delta)$.

\item[(7)]
In particular, if $(Y, \Gamma)$ is kawamata log terminal, then $F_1$ is rationally connected.

\end{itemize}
\end{thm}

We give a detailed proof of Theorem \ref{p-thm5.1}. 

\begin{proof}[Proof of Theorem \ref{p-thm5.1}]
Throughout this proof, we freely shrink $S$ around the point $s$ without further mention. 
Let $g\colon Y \to X$ be a resolution of singularities such that
$\pi := f \circ g\colon Y \to S$ is projective and both
$\pi^{-1}(s)$ and $\Exc(g)$ are simple normal crossing divisors on $Y$ 
(see \cite[Section 13]{bierstone-milman}). Let $\mathfrak m_s$ denote the maximal ideal corresponding to $s \in S$, 
and let $\mu\colon S' \to S$ be the blow-up of $S$ along $\mathfrak m_s$. 
We may further assume that $\pi\colon Y \to S$ factors through $\mu$ and that
\[
\Exc(g) \cup \Supp g^{-1}_*\Delta \cup \pi^{-1}(s)
\]
is a simple normal crossing divisor on $Y$. Finally, let $X' \to X$ be a bimeromorphic morphism. 
Then, by \cite[Corollary 2]{hironaka} and \cite[Section 13]{bierstone-milman}, we can arrange 
for $g\colon Y \to X$ to factor through $X' \to X$.

Let $F$ be the union of the irreducible components of $\pi^{-1}(s)$
whose discrepancies with respect to $K_X+\Delta$ are greater than $-1$.
Write
\[
K_Y + \Gamma' = g^*(K_X + \Delta) + E',
\]
where $\Gamma'$ and $E'$ are effective $\mathbb{Q}$-divisors with no
common irreducible components, $g_*\Gamma' = \Delta$, and $E'$ is
$g$-exceptional. Let $E^\dagger := \Exc(g)$ denote the reduced exceptional divisor.
For a sufficiently small rational number $\delta > 0$, set
\[
\Gamma'' := \Gamma' + \delta E^\dagger, \qquad
E'' := E' + \delta E^\dagger.
\]
Then $\Nklt(Y,\Gamma') = \Nklt(Y,\Gamma'')$. Since
\[
\Supp \Gamma'' = \Supp g^{-1}_* \Delta \cup \Exc(g),
\]
and $\Delta$ is $f$-big by assumption, 
we may write
\[
\Gamma'' \sim_{\mathbb{Q},\pi} A' + B',
\]
where $A'$ is a $\pi$-ample $\mathbb{Q}$-divisor, 
$B'$ is an effective $\mathbb{R}$-divisor, 
and $A'$ and $E''$ have no common irreducible components. 
Moreover, by choosing the resolution $g\colon Y \to X$ suitably 
(see \cite[Section 13]{bierstone-milman}), we may further assume that
\[
\Supp(A' + B') \cup \Supp g^{-1}_* \Delta
\cup \Exc(g) \cup \pi^{-1}(s)
\]
is a simple normal crossing divisor on $Y$.

Consider
\[
K_Y + \bigl((1-\varepsilon)\Gamma'' + \varepsilon B'\bigr)
+ \varepsilon A'
\sim_{\mathbb{Q},\pi}
g^*(K_X+\Delta) + E''
\] 
for $0<\varepsilon \ll 1$. 
Canceling common components on both sides yields
\[
K_Y + \Gamma \sim_{\mathbb{Q},\pi}
g^*(K_X+\Delta) + E,
\]
where $\Gamma$ and $E$ are effective $\mathbb{Q}$-divisors with no
common irreducible components, $E$ is $g$-exceptional, and
$\varepsilon A' \le \Gamma$.
Since $g^*(K_X+\Delta) \sim_{\mathbb{Q},\pi} 0$, we obtain
\[
K_Y + \Gamma \sim_{\mathbb{Q},\pi} E.
\]
By construction, $\Supp(\Gamma+E)$ is a simple normal crossing divisor.
Set $A := \varepsilon A'$ and $B := \Gamma - \varepsilon A'$ so that
$\Gamma = A + B$. 
We note that we can freely replace $A$ by a general effective
$\mathbb{Q}$-divisor
\[
\widetilde{A}\sim_{\mathbb{Q},\pi}A.
\]
More precisely, let $m$ be a sufficiently large and divisible positive
integer, and let $A^\dag$ be a general effective Cartier divisor satisfying
\[
A^\dag\sim_{\pi}mA.
\]
Then
\[
\widetilde{A}:=\frac{1}{m}A^\dag
\]
is a general effective $\mathbb{Q}$-divisor
$\mathbb{Q}$-linearly equivalent to $A$ over $\pi$. 
Moreover, we have 
\[
\Nklt(Y,\Gamma)
\subset \Nklt(Y,\Gamma'')
= \Nklt(Y,\Gamma')
\subset g^{-1}\Nklt(X,\Delta),
\]
where the first inclusion holds by $0<\varepsilon\ll 1$.
In particular, if $(X,\Delta)$ is kawamata log terminal,
then so is $(Y,\Gamma)$.

We write $F = \sum_{i=1}^k F_i$ for its decomposition into irreducible components. 
We note that we may slightly perturb the coefficients of those $F_i$ 
that appear in $B$ or $E$ without changing $\Nklt(Y,\Gamma)$, as follows. 
For sufficiently small positive rational numbers $\varepsilon_i$, 
we replace
$
B$ and $A$ by
$B + \sum_{i=1}^k \varepsilon_i F_i$ and  
$A - \sum_{i=1}^k \varepsilon_i F_i$, 
respectively. 
After subtracting the common irreducible components of 
$B + \sum_{i=1}^k \varepsilon_i F_i$ and $E$, 
and replacing 
$A - \sum_{i=1}^k \varepsilon_i F_i$ 
with a general effective $\mathbb{Q}$-divisor 
that is $\mathbb{Q}$-linearly equivalent to it over $S$, 
the non-klt locus $\Nklt(Y,\Gamma)$ remains unchanged.

Let $\sigma_1, \ldots, \sigma_m$ be generators of the maximal ideal
$\mathfrak m_s \subset \mathcal O_S$.
Since $\pi\colon Y \to S$ factors through the blow-up
$\mu\colon S' \to S$ of $S$ along $\mathfrak m_s$, there exists an
effective Cartier divisor $M$ on $Y$ such that
\[
\pi^{-1}\mathfrak m_s\cdot \mathcal O_Y=\mathcal O_Y(-M),
\]
where $\Supp M=\pi^{-1}(s)$.
The pull-backs
$\pi^*\sigma_1,\ldots,\pi^*\sigma_m$
generate $\mathcal O_Y(-M)$.
Hence, by Bertini's theorem, we may choose sufficiently many general
effective Cartier divisors
$H_1,\ldots,H_l$ on $S$ passing through $s$, where $l\gg0$, such that,
setting
\[
G:=\frac{1}{2}\sum_{j=1}^{l}H_j,
\]
the divisor
\[
\Supp(\Gamma+E)\cup\Supp(\pi^*G)\cup\pi^{-1}(s)
\]
is a simple normal crossing divisor on $Y$.

For $t \in [0,1]$, set $\Delta_t := \Delta + t f^*G$.
Let $\Gamma_t$ and $E_t$ be the effective $\mathbb Q$-divisors
obtained from $\Gamma + t\pi^*G$ and $E$, respectively, 
by subtracting their common irreducible components so that
\[
K_Y + \Gamma_t \sim_{\mathbb{Q},\pi} E_t.
\]
Let $V_t$ be the closure of
\[
\Nklt(Y,\Gamma_t)\setminus\Nklt(Y,\Gamma).
\]
By construction, we have $F = V_1$.
Since we may slightly perturb the coefficients of those $F_i$ 
that appear in $B$ or $E$, 
there exist rational numbers
\[
0 = t_0 < t_1 < \cdots < t_k \le 1
\]
such that, after renumbering the $F_i$ if necessary,
\[
V_{t_i} = F_1 \cup \cdots \cup F_i
\]
for each $i$.
Moreover, for any sufficiently small $0 < \varepsilon \ll 1$,
\[
V_{t_i - \varepsilon} = V_{t_{i-1}}.
\]

We have already verified (1)--(5). 
Assertion (6) follows from Lemma~\ref{p-lem5.2} below. 
Finally, since $F_1$ is rationally chain connected by (6) and is smooth, 
it is rationally 
connected (see Theorem \ref{p-thm2.11}). This proves (7) and completes the proof of 
Theorem~\ref{p-thm5.1}.
\end{proof}

The nontrivial part of Theorem \ref{p-thm5.1} is (6). 
This follows from the following lemma, whose proof is new 
and plays a central role in this paper.

\begin{lem}\label{p-lem5.2}
We use the same notation as in Theorem~\ref{p-thm5.1}.
Let $F^\circ$ be an irreducible component of $F$.
Then, by~(5) of Theorem~\ref{p-thm5.1}, there exists
$\tau^\circ\in(0,1]$ such that $F^\circ$ is a log canonical center of
$(Y,\Gamma_{\tau^\circ})$.
With this notation, for any sufficiently small positive rational number
$\varepsilon$, $F^\circ$ is rationally chain connected modulo
\[
W^\circ
:=
F^\circ\cap\Nklt(Y,\Gamma_{\tau^\circ-\varepsilon}).
\]
\end{lem}

\begin{proof}[Proof of Lemma \ref{p-lem5.2}]
We run a $(K_Y+\Gamma^b_{\tau^\circ})$-minimal model program over $X$ and then apply Lemma \ref{p-lem4.6} or Lemma \ref{p-lem4.8}. 

\setcounter{step}{0}
\begin{step}[Reduction to the non-klt locus on $F^\circ$]\label{p-step5.2.1}
Observe that 
\[
K_Y + \Gamma_{\tau^\circ} \sim_{\mathbb{Q},\pi} E_{\tau^\circ}, 
\]
and $F^\circ$ is an irreducible component of $\Gamma_{\tau^\circ}$. By adjunction, we write
\[
K_{F^\circ} + \Psi := (K_Y + \Gamma_{\tau^\circ})|_{F^\circ}, \quad \text{i.e., } \Psi = (\Gamma_{\tau^\circ} - F^\circ)|_{F^\circ}.
\] 
Then for any sufficiently small positive rational number $\varepsilon$, we have
\[
\Nklt(F^\circ, \Psi) = \Nklt(F^\circ, (\Gamma_{\tau^\circ-\varepsilon})|_{F^\circ}) = F^\circ \cap \Nklt(Y, \Gamma_{\tau^\circ-\varepsilon}).
\] 
Hence it suffices to show that $F^\circ$ is rationally chain connected modulo $\Nklt(F^\circ, \Psi)$.
\end{step}

\begin{step}[Running a $(K_Y + \Gamma_{\tau^\circ}^b)$-MMP over $X$]\label{p-step5.2.2}
Recall that 
\[
K_Y + \Gamma_{\tau^\circ}^b \sim_{\mathbb{Q},\pi} E_{\tau^\circ} - \Gamma_{\tau^\circ}^{nb},
\]
and that $A \le \Gamma_{\tau^\circ}^b$ by construction. 

Take an open neighborhood $U$ of $s$ and a Stein compact subset 
$W \subset S$ containing $U$ such that $\Gamma(W,\mathcal O_S)$ 
is noetherian. 
Then, by \cite[Lemma 9.4]{fujino-minimal}, we can run a 
$(K_Y + \Gamma_{\tau^\circ}^b)$-minimal model program over $X$ 
around $f^{-1}(W)$ with ample scaling, obtaining a finite sequence 
of flips and divisorial contractions:
\[
Y =: Y_0 \dashrightarrow Y_1 \dashrightarrow \cdots \dashrightarrow Y_m,
\]
such that the strict transform $(F^\circ)_m$ of $F^\circ$ remains a divisor on $Y_m$. We note that, at each step of the minimal model program,
it is necessary to shrink the spaces further around $W$. For any $\mathbb R$-divisor $D$ on $Y$, we denote by $(D)_m$ 
its pushforward on $Y_m$.

We distinguish two cases:
\begin{itemize}
\item[(A)] $K_{Y_m} + (\Gamma_{\tau^\circ}^b)_m$ is nef over $f^{-1}(W)$, or
\item[(B)] $Y_m$ admits a $(K_{Y_m} + (\Gamma_{\tau^\circ}^b)_m)$-negative divisorial contraction that contracts $(F^\circ)_m$.
\end{itemize}

By adjunction, set
\[
K_{(F^\circ)_m} + \Phi := (K_{Y_m} + (\Gamma_{\tau^\circ})_m)|_{(F^\circ)_m}.
\] 
Then one verifies that 
\[
K_{(F^\circ)_m} + \Phi^b = (K_{Y_m} + (\Gamma_{\tau^\circ}^b)_m)|_{(F^\circ)_m}.
\] 
Moreover, $((F^\circ)_m, \Phi^b)$ is divisorial log terminal because $(Y_m, (\Gamma_{\tau^\circ}^b)_m)$ is divisorial log terminal.  

Furthermore, $(K_{Y_m} + (\Gamma_{\tau^\circ})_m)$ is $\mathbb Q$-linearly equivalent to $(E_{\tau^\circ})_m$ over $S$, and $(E_{\tau^\circ})_m$ shares no components with $(F^\circ)_m$, so
\[
\kappa((F^\circ)_m, K_{(F^\circ)_m} + \Phi) \ge 0.
\] 
Since $\Gamma_{\tau^\circ}^b \ge A$, $\Phi^b$ contains an effective big $\mathbb{Q}$-Cartier $\mathbb{Q}$-divisor.
\end{step}

\begin{step}[Case (A)]\label{p-step5.2.3}
In this step, we treat Case~(A).

In Case (A), the negativity lemma implies $(E_{\tau^\circ})_m = 0$ 
over $U$, hence
\[
K_{(F^\circ)_m} + \Phi \sim_{\mathbb{Q}} 0 \quad \text{and} \quad 
\kappa((F^\circ)_m, K_{(F^\circ)_m} + \Phi^b) \le \kappa((F^\circ)_m, K_{(F^\circ)_m} + \Phi) = 0.
\] 
Thus $((F^\circ)_m, \Phi)$ satisfies conditions (a)--(c) of Lemma \ref{p-lem4.6}.

Let $q \colon V \to (F^\circ)_m$ be a resolution such that
\begin{equation}\label{p-eq5.1}
K_V + \Delta_V = q^*(K_{(F^\circ)_m} + \Phi) + E^\sharp,
\end{equation}
as in Lemma~\ref{p-lem4.6}. 
Then Lemma~\ref{p-lem4.6} implies that $V$ is rationally chain connected modulo $\Nklt(V, \Delta_V)$.

We may assume that $V$ is a common resolution of $F^\circ$ and $(F^\circ)_m$:
\[
\xymatrix{
& V \ar[dl]_p \ar[dr]^q & \\
F^\circ && (F^\circ)_m.
}
\]

Let $\nu$ be any irreducible component of $\Delta_V^{\ge 1}$. 
Then
\[
a(\nu, V, \Delta_V - E^\sharp)
=
a(\nu, V, \Delta_V)
\le -1.
\]
Hence, by \eqref{p-eq5.1}, we obtain
\[
a(\nu, (F^\circ)_m, \Phi)\le -1.
\]

Since
\[
K_{(F^\circ)_m}+\Phi
=
\left(K_{Y_m}+(\Gamma_{\tau^\circ})_m\right)|_{(F^\circ)_m}
\sim_{\mathbb Q} 0,
\]
it follows that
\[
a(\nu, F^\circ, \Psi-(E_{\tau^\circ})|_{F^\circ})
=
a(\nu, (F^\circ)_m, \Phi)
\le -1.
\]
Therefore
\[
a(\nu, F^\circ, \Psi)\le -1.
\]
This shows that
\[
p(\Nklt(V, \Delta_V))
=
p(\Supp \Delta_V^{\ge 1})
\subset \Nklt(F^\circ, \Psi).
\]
Hence $F^\circ$ is rationally chain connected modulo $\Nklt(F^\circ, \Psi)$.
\end{step}

\begin{step}[Case (B)]\label{p-step5.2.4}
In this step, we treat Case~(B).

In Case (B), $(F^\circ)_m$ is contracted by a $(K_{Y_m} + (\Gamma^b_{\tau^\circ})_m)$-negative divisorial contraction. 
By the $(K_{Y_m}+(\Gamma^b_{\tau^\circ})_m)$-negative divisorial contraction 
contracting $(F^\circ)_m$, we obtain a projective surjective morphism 
\[
(F^\circ)_m \to W
\]
with connected fibers such that $-(K_{(F^\circ)_m}+\Phi^b)$ is ample over $W$ and 
$\dim W < \dim (F^\circ)_m$. 
Since
\[
K_{(F^\circ)_m}+\Phi^b
\sim_{\mathbb Q}
(E_{\tau^\circ})_m|_{(F^\circ)_m}-\Phi^{nb},
\] 
we obtain that $\Supp \Phi^{nb}$ dominates $W$. Hence, 
the non-klt locus $\Nklt((F^\circ)_m, \Phi^b)$ dominates $W$.

Let $q \colon V \to (F^\circ)_m$ be a resolution such that
\[
K_V + \Delta_V
=
q^*(K_{(F^\circ)_m} + \Phi^b) + E^\flat,
\]
as in Lemma~\ref{p-lem4.8}. 
Then Lemma~\ref{p-lem4.8} implies that $V$ is rationally chain connected modulo $\Nklt(V, \Delta_V)$.

We may assume that $V$ is a common resolution of $F^\circ$ and $(F^\circ)_m$:
\[
\xymatrix{
& V \ar[dl]_p \ar[dr]^q & \\
F^\circ && (F^\circ)_m.
}
\]

By the negativity lemma, for any divisor $\nu$ over $F^\circ$ we have
\[
a(\nu, (F^\circ)_m, \Phi^b)
\ge
a(\nu, F^\circ, \Psi^b).
\] 
Hence, if $a(\nu, V, \Delta_V)\leq -1$, then
\[
a(\nu, F^\circ, \Psi^b)
\le
a(\nu, (F^\circ)_m, \Phi^b)
=
a(\nu, V, \Delta_V-E^\flat)
=
a(\nu, V, \Delta_V)
\le
-1.
\]
It follows that
\[
p(\Nklt(V, \Delta_V))
\subset
\Nklt(F^\circ, \Psi^b)
=
\Nklt(F^\circ, \Psi).
\]
Thus $F^\circ$ is rationally chain connected modulo $\Nklt(F^\circ, \Psi)$.
\end{step}

This completes the proof of Lemma \ref{p-lem5.2}.
\end{proof}

Here, we make a remark on the treatment of our Theorem \ref{p-thm5.1}, 
which is essentially \cite[Theorem 5.1]{hacon-mckernan}, 
and point out the differences from the way Hacon and 
M\textsuperscript{c}Kernan 
handle the corresponding result.

\begin{rem}\label{p-rem5.3}
Although we work with complex analytic spaces rather than algebraic
varieties, Theorem~\ref{p-thm5.1} is essentially the same as
\cite[Theorem~5.1]{hacon-mckernan}.
However, we do not address the statement concerning the Kodaira
dimension in \cite[Theorem~5.1~(4)]{hacon-mckernan}.
The proof of \cite[Theorem~5.1~(4)]{hacon-mckernan} relies on a highly
nontrivial extension theorem; see
\cite[Theorem~5.2]{hacon-mckernan} and
\cite[Corollary~3.17]{hacon-mckernan-b} for details.
Since this extension result is technically demanding even for
specialists, we avoid using it in the present paper.

Instead, we prove Theorem~\ref{p-thm5.1} by running the minimal model
program (see \cite{fujino-minimal}).
Of course, the existence of flips itself ultimately depends on
extension theorems, and therefore our approach does not yield a
substantial simplification of the overall theory.
Nevertheless, we find the minimal model program more accessible and
conceptually transparent than working directly with the highly
technical extension results.
\end{rem}

\begin{rem}\label{p-rem5.4}
Roughly speaking, Hacon and M\textsuperscript{c}Kernan directly prove that 
$\kappa _{\sigma} (F^\circ, K_{F^\circ}+\Psi^b)\leq 0$ by using their difficult extension theorem (see \cite[Corollary 3.17]{hacon-mckernan-b}). 
They say that it is the tricky part of their proof of Theorem \ref{p-thm1.1} (see \cite[Theorem 1.2]{hacon-mckernan}). In contrast, in our proof of Lemma \ref{p-lem5.2}, 
it may happen that
\[
\kappa(F^\circ, K_{F^\circ} + \Psi^b)
>
\kappa\bigl((F^\circ)_m, K_{(F^\circ)_m} + \Phi^b\bigr).
\]
Thus our method does not allow us to directly deduce 
$\kappa(F^\circ, K_{F^\circ} + \Psi^b) \le 0$. 
Instead, we prove that a resolution of 
$((F^\circ)_m, \Phi)$ (resp.~$((F^\circ)_m, \Phi^b)$) 
is rationally chain connected modulo the non-klt locus. 
It then follows that the same property holds for 
$(F^\circ, \Psi)$ (resp.~$(F^\circ, \Psi^b)$). 
\end{rem}

\section{Proofs of the Results}\label{p-sec6}

In this section, we prove the results stated in Section~\ref{p-sec1}.

\begin{proof}[Proof of Theorem \ref{p-thm1.1}]
By taking the Stein factorization, we may assume that 
$f_*\mathcal O_X \simeq \mathcal O_S$. 
Fix a point $P \in S$ and shrink $S$ around $P$ 
without further comment. 
Since $-(K_X+\Delta)$ is $f$-semiample by assumption, 
we take an effective $\mathbb{R}$-divisor $D_1$ such that 
\[
-(K_X+\Delta) \sim_{\mathbb{R},f} D_1
\]
and 
\[
\Nklt(X,\Delta) = \Nklt(X,\Delta_1),
\]
where $\Delta_1 := \Delta + D_1$. 
Since $\Delta_1 \sim_{\mathbb{R},f} -K_X$ and is $f$-big, 
we write
\[
\Delta_1 \sim_{\mathbb{Q},f} A + B,
\]
where $A$ is an effective $f$-ample $\mathbb{Q}$-divisor 
and $B$ is an effective $\mathbb{R}$-divisor. 
Set
\[
\Delta_2 := (1-\varepsilon)\Delta_1 + \varepsilon B
\]
for some $0 < \varepsilon \ll 1$. 
Then
\[
-(K_X+\Delta_2) \sim_{\mathbb{R},f} \varepsilon A,
\]
which is $f$-ample, and
\[
\Nklt(X,\Delta_2)
\subset 
\Nklt(X,\Delta_1)
=
\Nklt(X,\Delta).
\]
By slightly perturbing $\Delta_2$, we obtain an effective 
$\mathbb{Q}$-divisor $\Delta_3$ such that 
$-(K_X+\Delta_3)$ is $f$-ample and 
\[
\Nklt(X,\Delta_3)
=
\Nklt(X,\Delta_2)
\subset
\Nklt(X,\Delta).
\]
Take an effective $\mathbb{Q}$-divisor $D_2$ such that 
\[
-(K_X+\Delta_3) \sim_{\mathbb{Q},f} D_2
\]
and
\[
\Nklt(X,\Delta_4)
=
\Nklt(X,\Delta_3)
\subset
\Nklt(X,\Delta),
\]
where $\Delta_4 := \Delta_3 + D_2$. 
Set $\Delta^\dag := \Delta_4$. 
Then
\[
K_X+\Delta^\dag \sim_{\mathbb{Q},f} 0,
\]
$\Delta^\dag$ is $f$-big, and 
\[
\Nklt(X,\Delta^\dag) \subset \Nklt(X,\Delta).
\] 
Therefore, we can apply Theorem \ref{p-thm5.1}. 
By Lemma \ref{p-lem2.12}, we may replace $Y$ by a higher model if necessary. By Theorem \ref{p-thm5.1} (6), 
$\pi^{-1}(P)$ is rationally chain connected modulo 
$g^{-1}\Nklt(X,\Delta^\dag)$. 
Consequently, it is rationally chain connected modulo
$g^{-1}\Nklt(X,\Delta)$.
\end{proof}

\begin{proof}[Proof of Corollary \ref{p-cor1.2}]
Throughout this proof, we shrink $X$ around a point $P$ 
without further comment. By perturbing the coefficients of $\Delta$, 
we may assume that $(X,\Delta)$ is kawamata log terminal 
and that $K_X+\Delta$ is $\mathbb{Q}$-Cartier. 
We apply Theorem \ref{p-thm1.1} to the identity map 
$f\colon X \to X$. 
Then any fiber of $g\colon Y \to X$ is rationally chain connected. 
This proves (i). 

For (ii), if $X$ is rationally connected, then it is clearly 
rationally chain connected. 
Thus, it suffices to prove the ``only if'' part. 
Let $g\colon Y \to X$ be a resolution of singularities. 
By (i), $g$ has rationally chain connected fibers, 
and hence $Y$ is rationally chain connected. 
Since $Y$ is smooth, it follows that $Y$ is rationally connected 
(see Theorem \ref{p-thm2.11}). 
Therefore, $X$ is rationally connected. 
This completes the proof.
\end{proof}

\begin{proof}[Proof of Corollary \ref{p-cor1.3}]
We apply Corollary \ref{p-cor1.2} (i) to 
$p\colon \Gamma \to X$. 
Then $p^{-1}(x)$ is rationally chain connected. 
Thus, by Lemma \ref{p-lem2.12}, 
$q(p^{-1}(x))$ is rationally chain connected. 
\end{proof}

\begin{proof}[Proof of Corollary \ref{p-cor1.4}]
Let $\Gamma$ be the graph of 
$f \colon X \dashrightarrow Y$, 
and denote by $p \colon \Gamma \to X$ the natural projection. 
By Corollary~\ref{p-cor1.2} (i), 
every positive-dimensional fiber of $p$ 
is rationally chain connected. 
Since $Y$ contains no rational curves, 
every fiber of $p$ must be zero-dimensional. 
Hence $p$ is an isomorphism, 
and therefore $f$ is a morphism.
\end{proof}

\begin{proof}[Proof of Theorem \ref{p-thm1.6}]
By perturbing $\Delta$ slightly, we may assume that
$-(K_X+\Delta)$ is an ample $\mathbb{Q}$-divisor while preserving
$\Nklt(X,\Delta)$.
Let $D$ be a general effective ample $\mathbb{Q}$-divisor with
\[
D\sim_{\mathbb Q}-(K_X+\Delta),
\]
and set $\Delta^\dag:=\Delta+D$.
Then
\[
K_X+\Delta^\dag\sim_{\mathbb Q}0,
\]
$(\Delta^\dag)^b\ge D$, and
\[
\Nklt(X,\Delta)=\Nklt(X,\Delta^\dag).
\]
By applying Lemma \ref{p-lem4.6} to a resolution of $X$, we conclude that
$X$ is rationally chain connected modulo $\Nklt(X,\Delta)$.
If $(X,\Delta)$ is kawamata log terminal, then
$X$ is rationally connected by Corollary \ref{p-cor4.7}.
\end{proof}

\begin{proof}[Proof of Theorem \ref{p-thm1.7}]
This theorem is a corollary of Theorem \ref{p-thm1.6}. 
In the algebraic setting, the proof is given in 
\cite[Section 13]{fujino-hyp}. 
Note that Theorem \ref{p-thm1.1}, and therefore Theorem \ref{p-thm5.1}, are not needed in this proof.
In the complex analytic setting, the same argument applies, 
since the theory of quasi-log structures is established in 
\cite{fujino-quasi-log}. 
\end{proof}

\begin{proof}[Proof of Corollary \ref{p-cor1.8}]
By taking the Stein factorization, we may assume that 
$f_*\mathcal O_X \simeq \mathcal O_S$. 
Throughout this proof, we fix a point $s \in S$ 
and shrink $S$ around $s$ without further comment. 
If $(X,\Delta)$ is log canonical, 
then $[X,K_X+\Delta]$ naturally becomes a quasi-log canonical pair. 
Hence, (2) is a special case of Theorem \ref{p-thm1.7}. 
For (1), as in the proof of Lemma \ref{p-lem2.7}, 
we construct an effective $\mathbb{R}$-divisor $\Theta$ 
such that $(X,\Theta)$ is kawamata log terminal and 
$-(K_X+\Theta)$ is $\pi$-ample. 
Thus, (1) follows from (2). 
This completes the proof. 
\end{proof}

For the reader's convenience, we present an alternative proof 
of Corollary \ref{p-cor1.8} (1) using Theorem \ref{p-thm1.1}. 
This argument follows the strategy of 
Hacon and M\textsuperscript{c}Kernan in \cite{hacon-mckernan}. 

\begin{proof}[Proof of Corollary \ref{p-cor1.8} (1) as an application of 
Theorem \ref{p-thm1.1}] 
As usual, we fix a point $P \in S$ and shrink $S$ around $P$ 
without further comment. 
By the basepoint-free theorem (see Lemma \ref{p-lem2.7}), 
$-(K_X+\Delta)$ is $f$-semiample. 
Since $\Nklt(X,\Delta)=\emptyset$, 
Theorem \ref{p-thm1.1} together with Lemma \ref{p-lem2.12} 
implies that $f^{-1}(P)$ is rationally chain connected. 
\end{proof}

%%%%%%%%%%%%%%%%%%%%

\end{document}